\documentclass{article}
\usepackage{enumerate}

\usepackage{amsmath}
\usepackage{amsfonts}
\usepackage{amssymb}
\usepackage[english]{babel}
\usepackage{graphicx,epsf,cite,esint}
\usepackage{amsthm}
\usepackage{mathtools}

\usepackage[utf8]{inputenc}

\usepackage{amsmath,amsfonts,amssymb,amsthm,cases, hyperref}

\usepackage{color}
\usepackage{vmargin}
\usepackage{dsfont}
\usepackage{esint} 

\setmarginsrb{2cm}{1cm}{2cm}{3cm}{1cm}{1cm}{2cm}{2cm}

\newcommand{\R}{\mathbb{R}}
\newcommand{\C}{\mathbb{C}}

\renewcommand{\div}{\operatorname{div}}
\newcommand{\curl}{\operatorname{curl}}

\newcommand{\supp}{\operatorname{supp}}

\renewcommand{\leq}{\leqslant}\renewcommand{\le}{\leqslant}
\renewcommand{\geq}{\geqslant}\renewcommand{\ge}{\geqslant}
\newtheorem{theorem}{Theorem}[section]

\newtheorem{proposition}[theorem]{Proposition}
\newtheorem{lemma}[theorem]{Lemma}

\theoremstyle{remark}

\numberwithin{equation}{section}
\usepackage{enumitem}
\makeatletter
\def\namedlabel#1#2{\begingroup
 #2%
 \def\@currentlabel{#2}%
 \phantomsection\label{#1}\endgroup
}
\makeatother

\DeclareMathOperator*{\sgn}{sgn}

\allowdisplaybreaks
\numberwithin{equation}{section}

\newcommand{\lb}{\label}

\newcommand{\beq}{\begin{equation}}
\newcommand{\eeq}{\end{equation}}

\newcommand{\bal}{\begin{align}}
\newcommand{\eal}{\end{align}}
\newcommand{\bals}{\begin{align*}}
\newcommand{\eals}{\end{align*}}

\newcommand{\bbR}{{\mathbb{R}}}

\newcommand{\calT}{{\mathcal T}}

\newcommand{\til}{\tilde}

\title{The Euler Equations in Planar Domains with Corners} 
\author{Christophe Lacave and Andrej Zlato\v s}
\def\adrese{
\begin{description}
\item[C. Lacave:] Univ. Grenoble Alpes, CNRS, Institut Fourier, F-38000 Grenoble, France.\\
Email: \texttt{christophe.lacave@univ-grenoble-alpes.fr}\\
Web page: \texttt{\url{https://www-fourier.ujf-grenoble.fr/~lacavec/}}

\item[A. Zlato\v s:] Department of Mathematics, University of California San Diego, La Jolla, CA 92093, USA\\
Email: \texttt{zlatos@ucsd.edu}\\
Web page: \texttt{\url{http://math.ucsd.edu/~zlatos/}}
\end{description}
}

\date{\today}
\begin{document}
\maketitle

\begin{abstract}
When the velocity field is not a priori known to be globally almost Lipschitz, global uniqueness of solutions to the two-dimensional Euler equations has been established only in some special cases, and the solutions to which these results apply share the property that the diffuse part of the vorticity is constant near the points where the velocity is insufficiently regular. Assuming that the latter holds initially, the challenge is then to propagate this property along the Euler dynamic via an appropriate control of the Lagrangian trajectories. In domains with obtuse corners and sufficiently smooth elsewhere, Yudovich solutions fail to be almost Lipschitz only near these corners, and we investigate necessary and sufficient conditions for the vorticity to remain constant there. We show that if the vorticity is initially constant near the whole boundary, then it remains such forever (and global weak solutions are unique), provided no corner has angle greater than $\pi$. We also show that this fails in general for domains that do have such corners.
\end{abstract}

%\tableofcontents % TO DELETE

%\nocite{*} % TO DELETE

\section{Introduction}

The most celebrated equations modeling the motion of an adiabatic and inviscid flow are undoubtedly the Euler equations. Fluid velocity $u$ and pressure $p$ on a spatial domain $\Omega$ are related via the equation
\begin{equation}\label{eq.Euler1}
 \partial_{t} u + (u\cdot \nabla) u = -\nabla p \qquad \text{on } (0,\infty) \times \Omega,
\end{equation}
and in the simplest setting, the fluid is assumed to have a constant density and to be incompressible:
\begin{equation}\label{eq.Euler2}
 \div u=0 \qquad \text{on } [0,\infty)\times \Omega.
 \end{equation}
 Moreover, in domains with boundaries, it is natural to assume that the boundary is impermeable:
 \begin{equation}\label{eq.Euler3}
u\cdot n =0 \qquad \text{on } [0,\infty)\times \partial\Omega.
 \end{equation}
Even though these equations are the first PDE model for the motion of a liquid, their study is still a very active area of research, in mathematics as well as in engineering and physics, because they not only describe well the motion of perfect fluids, but the structure of this system is also incredibly rich. 

In two spatial dimensions, the case considered in this article, the vorticity 
 \begin{equation}\label{eq.Euler4}
 \omega:=\curl u=\partial_{1} u_{2}-\partial_{2}u_{1} \qquad \text{on } [0,\infty)\times \Omega
 \end{equation}
 plays a crucial role, which one can see after taking the curl of Equation \eqref{eq.Euler1} to obtain
 \begin{equation}\label{eq.Euler5}
 \partial_{t} \omega + u\cdot \nabla \omega = 0 \qquad \text{on } (0,\infty) \times \Omega.
\end{equation}
The velocity formulation \eqref{eq.Euler1}--\eqref{eq.Euler3} is then equivalent to the vorticity formulation \eqref{eq.Euler2}--\eqref{eq.Euler5}, that is, the Euler system can be viewed as a transport equation \eqref{eq.Euler5} for the vorticity $\omega$, with the advecting velocity field $u$ obtained in terms of $\omega$ by solving the div-curl problem \eqref{eq.Euler2}--\eqref{eq.Euler4}.

\subsection{Existing well-posedness results}

Thanks to several conservation properties coming from the transport theory, well-posedness results for strong solutions to Euler equations in two dimensions were established long time ago: by Wolibner \cite{Wolibner} in bounded domains, by McGrath \cite{McGrath} in the whole plane, and by Kikuchi \cite{Kikuchi} in exterior domains. The vorticity formulation allows us to look for more general solutions in spaces defined in terms of the regularity of $\omega_{0}:=\omega(0,\cdot)$. The most celebrated result in this direction is the work of Yudovich \cite{yudo}, in which he obtained existence and uniqueness of a global weak solution for $\omega_{0}\in L^1\cap L^\infty(\Omega)$ (see also \cite{Bardos,Temam}). Disregarding the uniqueness issue, one can even obtain existence of a global weak solution for $\omega_{0}\in L^1\cap L^p(\Omega)$ with $p>1$ \cite{DiPernaMajda} and for $\omega_{0}\in H^{-1}\cap \mathcal{M}_{+}(\Omega)$ \cite{Delort}.

Unfortunately, all the above works consider only smooth domains, with $\partial\Omega$ being at least $C^{1,1}$. This restriction is not justified by the weak regularity of the studied solutions, and leaves aside many situations of practical interest. Mathematically, smoothness of $\partial\Omega$ is used to deduce a priori estimates of $\nabla u=\nabla \nabla^\perp \Delta^{-1}\omega$ in terms of $\omega$ thanks to $L^p$-continuity of the Riesz transform for any $p\in (1,\infty)$. We mention that while Jerison and Kenig constructed an example of a simply-connected bounded domain $\Omega$ of class $C^1$ and a function $f\in C^\infty(\Omega)$ such that $D^2 \Delta^{-1}f$ is not integrable \cite{Kenig}, this problem does not arise for Leray solutions of the Navier-Stokes equations on non-smooth domains because the estimate of $\nabla u$ comes directly from the energy estimate (see, e.g., \cite{CDGG}).

If the domain is convex, then the Riesz transform is continuous on $L^2(\Omega)$, and Taylor used this to obtain global existence of weak solutions for the 2D Euler equations \cite{Taylor}. More recently, G\' erard-Varet and Lacave extended the existence theory to a large class of irregular domains \cite{GV-Lac,GV-Lac2}, even allowing exotic geometries such as the Koch snowflake. Both articles consider $\omega_{0}\in L^1\cap L^p(\Omega)$ or $\omega_{0}\in H^{-1}\cap \mathcal{M}_{+}(\Omega)$, using only $L^2$ estimates for the velocity.

To address the question of uniqueness, even in smooth domains, we typically need much more regularity for the velocity, namely almost Lipschitz. For instance, Yudovich used the Calder\'on-Zygmund inequality
\[
\| \nabla u \|_{L^p} \leq C p \| \omega \|_{L^p} \qquad \forall p\in [2,\infty)
\]
to perform a Gronwall type argument when $\omega\in L^\infty([0,\infty);L^1\cap L^\infty(\Omega))$. Alternatively, one can use the well-known log-Lipschitzness of the velocity associated to a vorticity in $L^1\cap L^\infty(\Omega)$. That is,
\[
|u(x)-u(y)| \leq C(\|\omega\|_{L^1\cap L^\infty}) |x-y| \max\{1, -\ln|x-y|\} \qquad \forall x,y\in \Omega,
\]
which then implies uniqueness for the Lagrangian formulation \cite{MarPul}. Below this level of velocity regularity, we are aware of only one example where we have uniqueness in smooth domains, the vortex-wave system when the diffuse part of the initial vorticity is constant near the point vortex. This system was introduced by Marchioro and Pulvirenti to describe the Euler solution when the total vorticity is composed of a regular part $\omega\in L^1\cap L^\infty(\Omega)$ and a concentrated part $\gamma\delta_{z(t)}$. After proving that the regular part $\omega$ stays constant around the point vortex $z$ \cite{MarPul-VWS}, which is the place where the velocity is not regular, it is possible to prove uniqueness \cite{lacave-miot}.
We also note that uniqueness may not hold for unbounded vorticities, even in smooth domains, as is suggested by the papers \cite{Vis1,Vis2} of Vishik, where non-uniqueness was demonstrated on $\mathbb R^2$ with $\omega(0,\cdot) \in L^p(\mathbb R^2)$ for some $p>2$ and in the presence of forcing that is (uniformly in time) in the same space.

When it comes to less regular domains, most existing uniqueness results require $u\in \bigcap_{p\geq2 } W^{1,p}(\Omega)$, while the example of Jerison and Kenig shows that this regularity cannot be reached for general $C^1$ domains. Nevertheless, there is a large literature on elliptic regularity in domains with corners, that is, such that $\partial\Omega$ is piecewise regular, with the singular points all being corners. In this case, very precise estimates on the solution to the Laplace problem are known (see, for instance, \cite{Kondra,Grisvard,Mazya}) and depend on the angles of the corners. In particular, we note that $D^2 \Delta^{-1}f\in \bigcap_{p\geq2 } W^{1,p}(\Omega)$ for any smooth function $f$ if and only if all the corners are acute. However, the Calder\'on-Zygmund inequality had not been established in this case and it was not clear how to use the above results to obtain uniqueness. 

Instead, other methods proved useful in recent years. Bardos, Di Plinio, and Temam proved uniqueness when $\Omega$ is a square, using a reflexion argument \cite{BDT}, which can also be extended to convex domains with angles of all corners being $\frac \pi{2^k}$ for some integers $k$. General domains with acute corners and $\partial\Omega\in C^{2,\alpha}$ (with $\alpha>0$) away from the corners were treated by the first author, Miot, and Wang in \cite{LMW}, where precise estimates on the relevant conformal mapping lead to a log-Lipschitz estimate on the push forward of $u$ onto the unit disk, followed by a version of the uniqueness proof of Marchioro and Pulvirenti. Afterwards, Di Plinio and Temam proved uniqueness for general domains with acute corners and $\partial\Omega\in C^{1,1}$ away from the corners \cite{DT}, via obtaining a Calder\'on-Zygmund inequality for such domains and then employing the argument of Yudovich.

If $\Omega$ has a corner with an obtuse angle, then the velocity $u$ is far from Lipschitz, and it is not even bounded if the angle is greater than $\pi$. So just as for the vortex-wave system, the question of global uniqueness appears to be a very challenging problem if the vorticity is not constant in the neighborhood of the singularities of the velocity (i.e., of the obtuse corners in this case). If, on the other hand, $\omega_0$ is constant in these regions, the natural question becomes under what conditions this remains the case at later times. This was addressed by the first author in \cite{Lacave-SIAM}, where he introduced a special Lyapunov function that allowed him to show that if $\omega_{0}$ is constant near all of $\partial\Omega$ (with $\partial\Omega\in C^{1,1}$ and $\Omega$ having no acute corners) and has a definite sign, then the same will be true for all times $t>0$.

The main purpose of this article is to show that the sign condition on $\omega$ as well as the requirement of no acute corners are superfluous when all corners have angles smaller than $\pi$, while the sign condition cannot be discarded in general when this is not the case.

\subsection{Main results}

We will assume here that $\Omega$ is a bounded simply connected open subset of $\R^2$, with a boundary that is $C^{1,\alpha}$ except at a finite number of corners.

\begin{description}
\item[\namedlabel{H}{\rm\bf(H)}] Assume that $\partial \Omega$ is a piecewise $C^{1,\alpha}$ Jordan curve with $\alpha>0$, that is, there is a bijection $\gamma:\mathbb T\to\partial\Omega$
which is $C^{1,\alpha}$ except at finitely many points $\{ s_{k}=\gamma^{-1}(x_k)\}_{k=1}^N$ and $|\gamma'(s)|=1$ for all $s\notin \{s_1,\dots,s_N\}$.
Also assume that $\gamma$ parametrizes $\partial\Omega$ in the counterclockwise direction (i.e., ${\rm Ind}_{\gamma}(z)\in \{0,1\}$ for each $z$) and all the singularities of $\partial\Omega$ are corners with positive angles. That is, for $k=1,\dots,N$ we have
\[ 
\theta_{k}:=\lim_{s\to 0_+} {\rm Angle}(\gamma'(s_k+s),-\gamma'(s_k-s))\in (0,2\pi].
\]
\end{description}

\noindent{\it Remark.}
The case $\theta_{k}=0$ corresponds to an exterior cusp, whereas $\theta_{k}=2\pi$ to an interior cusp. In the course of the proof, we will need to straighten the corner via the map $z\mapsto z^{\pi/\theta_{k}}$. As in other works (see, e.g., \cite{Pomm2}), we exclude the case $\theta_{k}=0$ in order to avoid complications arising from straightening exterior cusps (see, e.g., \cite[Section 1]{MazyaSolovev}).
\smallskip

We will consider here velocity fields $u$ in the Yudovich class, that is,
\begin{equation}\label{Yudo-class}
u\in L^\infty([0,\infty);L^2(\Omega)) \qquad\text{and}\qquad \omega:=\curl u \in L^\infty ([0,\infty)\times \Omega),
\end{equation}
that are weak solutions of the velocity or the vorticity formulation of the 2D Euler equations. Existence of such solutions is established in \cite{GV-Lac}, and we postpone their precise definitions to Section~\ref{Sect.3.2}. Since the velocity is far from Lipschitz when corners with obtuse angles are present, a crucial step in the study of uniqueness is to estimate particle trajectories near $\partial\Omega$. As will be explained in Section~\ref{Sect.2.2} below, local elliptic regularity shows that for any divergence-free vector field $u$ verifying \eqref{Yudo-class} and for any $x\in \Omega$, there exists $t(x)\in(0,\infty]$ and a unique curve $ X(\cdot,x)\in W^{1,\infty}([0,t(x)))$ with $X(0,x)=x$ such that $X(t,x)\in \Omega$ for each $t\in[0,t(x))$,
\begin{equation}\label{1.11}
\frac{d}{dt} X(t,x) = u(t,X(t,x)) \qquad \text{for almost every } t\in [0,t(x)),
\end{equation}
as well as $X(t(x),x)\in \partial\Omega$ if $t(x)<\infty$. Here $t(x)$ is the maximal time of existence of the trajectory inside $\Omega$.
The first author showed in \cite{Lacave-SIAM} that $t(x)=\infty$ for domains as above with $\alpha=1$ and no acute corners (i.e., $\min_{k} \theta_{k}\geq \frac \pi 2$) if either $\omega(0,\cdot)\ge 0$ or $\omega(0,\cdot)\le 0$. To get this result, he introduced a Lyapunov function based on the Green's function in order to obtain an algebraic cancelation of the singularities at the corners, and the sign condition was useful to show that this function essentially encodes the distance to the boundary (see Section~\ref{Sect.2.2} for more details on his approach).

Our first main result shows that the sign condition can be dropped for domains whose corners have angles less than $\pi$ (including acute ones).

\begin{theorem}\label{main1} 
Let $\Omega$ satisfy \ref{H} with $\max_{k} \theta_{k}< \pi$, and let $u$ be a global weak solution of the Euler equations on $\Omega$ from the Yudovich class \eqref{Yudo-class}. 

\noindent (i)
Then $t(x)=\infty$ for each $x\in\Omega$ and the corresponding trajectory $X(\cdot,x)$ from \eqref{1.11}.
\smallskip

\noindent (ii)
If $\alpha=1$, then $\omega(t,\cdot)=\omega(0,X^{-1}(t,\cdot))$ for all $t>0$ and \eqref{1.11} holds for all $(t,x)\in [0,\infty)\times\Omega$, with the velocity $u$ being continuous on $[0,\infty)\times \Omega$.
\smallskip

\noindent (iii) 
If $\alpha=1$ and there is $a\in\R$ such that $\omega(0,\cdot)-a$ is supported away from $\partial\Omega$, then the support of $\omega(t,\cdot)-a$ never reaches $\partial\Omega$, and $u$ is the unique global weak solution from the Yudovich class with the same $\omega(0,\cdot)$.
\end{theorem}

\noindent{\it Remarks.} 
1. To obtain $t(x)=\infty$ in (i), we introduce a simpler Lyapunov function than in \cite{Lacave-SIAM}, and the condition $\max_{k} \theta_{k}< \pi$ will be necessary to obtain a relevant Gronwall type estimate.
\smallskip

2. In fact, a more general version of (iii) holds, allowing $\supp \omega(0,\cdot)-a$ to meet $\partial\Omega$. Namely, Proposition~\ref{constant_vorticity} below gives uniqueness until the first time $t$ when the support of $\omega(t,\cdot)-a$ reaches a point where $\partial\Omega$ is not $C^{2,\tilde \alpha}$ for some $\tilde\alpha>0$ (e.g., a corner of $\Omega$). We conjecture that $C^{2,\tilde \alpha}$ can be replaced by $C^{1,1}$ here, that is, uniqueness holds as long as $\supp \omega(t,\cdot)-a$ vanishes near the corners. We also note that Proposition~\ref{constant_vorticity} holds for solutions satisfying (ii) on more general domains $\Omega$ than just those from \ref{H}.
\smallskip

Our second main result shows that the sign condition in \cite{Lacave-SIAM} is necessary for general domains whose corners have angles greater than $\pi$. (Note that if $\theta_k=\pi$ for some $k$, then $\partial\Omega$ is $C^{1,\alpha}$ at $x_k$, so $x_k$ is not a corner.)

\begin{theorem}\label{main2}
For any $\theta \in (\pi,2\pi]$, there exists $\Omega\subseteq\mathbb R^2$ satisfying \ref{H} with $\partial\Omega$ being $C^{\infty}$ except at one point, which is a corner with angle $\theta$, such that the following holds. There are weak solutions in the Yudovich class \eqref{Yudo-class} to the Euler equations on $\Omega$ such that $\omega(0,\cdot)$ is compactly supported inside $\Omega$ and there are infinitely many $x\in \supp(\omega(0,\cdot))$ such that the corresponding trajectory $X(\cdot,x)$ reaches $\partial\Omega$ (at the obtuse corner) in finite time.
\end{theorem}

\noindent{\it Remark.}
Our examples here are in spirit related to examples of solutions to the 2D Euler equations on domains with interior cusps that loose continuity in finite time, by Kiselev and the second author \cite{KisZla}.
We also provide in Section~\ref{sec.4} examples of solutions as in Theorem~\ref{main2}, but with the points $x$ reaching $\partial\Omega$ in finite time not belonging to the support of the vorticity at $t=0$, such that the Euler equations have a unique weak solution on $[0,\infty)\times\Omega$. In this case uniqueness holds because the corresponding trajectories do not transport vorticity. Uniqueness of Yudovich solutions when the vorticity is not constant near the obtuse corner remains an open question.
\smallskip

The remainder of this article is divided into four parts. In Section~\ref{sec.2} we recall an explicit formula for $u$ in terms of $\omega$ (the Biot-Savart law), expressed in terms of the Green's function on the unit disc via the Riemann mapping. We then obtain necessary estimates on the derivatives of the Riemann mapping close to the corners. We also review the approach from \cite{Lacave-SIAM}, based on a Lyapunov function. In Section~\ref{sec.3}, we prove Theorem~\ref{main1}, while Section~\ref{sec.4} is devoted to the construction of the examples from Theorem~\ref{main2}. A technical lemma, used in Section~\ref{sec.3}, is proved in Appendix~\ref{app-lemma}.

\bigskip

\noindent
{\bf Acknowledgements.} CL was partially supported by the CNRS program Tellus, by the Agence Nationale de la Recherche, Project IFSMACS, grant ANR-15-CE40-0010 and Project SINGFLOWS, grant ANR-18-CE40-0027-01. AZ acknowledges partial supported by NSF grant DMS-1652284.

\section{Preliminaries}\label{sec.2}

We start by recalling some basic results concerning the Biot-Savart law, which determines the velocity field $u$ from its vorticity $\omega$. An explicit formula for this law, in terms of the relevant conformal mapping, will be the key to building an appropriate Lyapunov function that controls the distance between particle trajectories and the boundary $\partial\Omega$. We will then discuss the construction and properties of the trajectories, and recall the strategy of the proof of the main result of \cite{Lacave-SIAM}.

\subsection{Riemann mapping and the Biot-Savart law}

To find the velocity field $u$ in the 2D Euler equations from its vorticity $\omega$, one needs to solve the div-curl problem
\[
\div u = 0 \text{ in } \Omega, \qquad \curl u = \omega \text{ in } \Omega, \qquad u\cdot n = 0 \text{ on } \partial\Omega.
\]
When the domain $\Omega$ is simply connected, for any $\omega \in H^{-1}(\Omega)$ there exists a unique solution $u\in L^2(\Omega)$ (in the sense of distributions; see \cite{Galdi,GV-Lac} for the weak tangency condition in non-smooth domains). Moreover, $u$ can be expressed via the stream function $\psi:= \Delta^{-1}\omega \in H^1_{0}(\Omega)$ (with $\Delta$ the Dirichlet Laplacian on $\Omega$) through the relation $u=\nabla^\perp \psi = (-\partial_{2} \psi, \partial_{1} \psi)$.

Identifying $\R^2$ with $\C$ by setting $z=x_{1}+ix_{2}$, one can invert the Laplacian via a biholomorphism $\mathcal{T}:\Omega\to D$ (which exists due to the Riemann mapping theorem), with $D$ the unit disk. As $\partial \Omega \in C^{0,1}$, the Kellogg-Warschawski Theorem (see \cite[Theorem 3.6]{Pomm2}) implies that $\mathcal{T}$ is continuous up to the boundary and maps $\partial \Omega$ on $\partial D$.
Using the form of the Green's function on $D$, we obtain the formula
\[
\psi(x)=\Delta^{-1}\omega (x)= \frac1{2\pi}\int_{\Omega} \ln\frac{| \mathcal{T}(x)- \mathcal{T}(y)|}{| \mathcal{T}(x)- \mathcal{T}(y)^*| | \mathcal{T}(y)|} \omega(y)\, dy,
\]
where $z^* := {z}{|z|^{-2}}$.
Therefore, the Biot-Savart law (for time-dependent functions) reads
\begin{equation}\label{Biot-Savart}
 u(t,x) = K_{\Omega}[\omega(t,\cdot)](x) : = \frac1{2\pi}D\mathcal{T}^T(x)\int_{\Omega} \left( \frac{ \mathcal{T}(x)- \mathcal{T}(y)}{| \mathcal{T}(x)- \mathcal{T}(y)|^2} - \frac{ \mathcal{T}(x)- \mathcal{T}(y)^*}{| \mathcal{T}(x)- \mathcal{T}(y)^*|^2 } \right)^\perp \omega(t,y)\, dy.
\end{equation}

Having this formula, it is natural to first analyze the regularity properties of $\mathcal{T}$. 

\begin{proposition} \label{prop T}
Let $\Omega$ satisfy \ref{H} with $\max_{k=1,\dots,N}\theta_{k} < \pi$. Let $\delta_{0}:=\frac{1}{6}\min_{i\neq j} \{|x_{i}-x_{j}|, |\mathcal{T}(x_{i})-\mathcal{T}(x_{j})|\})$. There exists $M\ge1$, depending only on $\Omega$, such that
 \begin{itemize}
 \item for all $x\in \Omega\setminus \bigcup_{k=1}^N B(x_{k},\delta_0)$ and $y\in D\setminus \bigcup_{k=1}^N B(\mathcal{T}(x_{k}),\delta_0)$ we have
 \begin{equation*}
M^{-1} \leq | D\mathcal{T}(x) | \leq M \qquad\text{and}\qquad M^{-1}\le | D \mathcal{T}^{-1}(y) | \leq M ;
\end{equation*}
 \item for any $k =1,\dots, N$ and all $x \in \Omega\cap B(x_{k},\delta_0)$ and $y\in D\cap B(\mathcal{T}(x_{k}),\delta_0)$ we have
\begin{align*}
M^{-1} |x-x_k|^{\pi/\theta_{k}-1} \leq & | D\mathcal{T}(x) | \leq M |x-x_k|^{\pi/\theta_{k}-1} , 
\\ M^{-1} |y-\mathcal{T}(x_{k}) |^{\theta_{k}/\pi-1}\le & | D \mathcal{T}^{-1}(y) | \leq M |y-\mathcal{T}(x_{k}) |^{\theta_{k}/\pi-1},
\\ M^{-1}|x-x_k|^{\pi/\theta_{k}} \le & | \mathcal{T}(x)-\mathcal{T}(x_{k}) | \leq M |x-x_k|^{\pi/\theta_{k}}.
\end{align*}
\end{itemize}
\end{proposition}

This proposition is very similar to \cite[Prop. 2.1]{LMW}, except that there $\partial\Omega$ needed to be piecewise $C^{2,\alpha}$ to guarantee certain properties of $D^2\mathcal{T}$. We provide the proof of our version for the convenience of the reader.

\begin{proof}
We identify $\R^2$ and $\C$ and we write here $\mathcal{T}'$ (i.e., derivative of $\calT:\C\to\C$) instead of $D\mathcal{T}$. Consider now the corner at $x_1$. 
The idea is to straighten it via the map $\varphi_1(z) := (z-x_{1})^{\pi/\theta_1}$. Since $\varphi_1$ need not be injective on $\Omega$, let $\delta_{1}\in (0,\frac 12\delta_{0}]$ be such that $\varphi_1$ is injective on $\Omega\cap B(x_{1},2\delta_{1})$. We next let $D_{1}\subseteq D$ be a $C^\infty$ Jordan domain such that 
\[
\Omega\cap B(x_{1}, \delta_{1}) \subset \mathcal{T}^{-1}(D_{1}) \subset \Omega \cap B(x_{1},2\delta_{1})
\]
and $g_{1}:D_1\to D$ be a Riemann mapping. From elliptic estimates on $\mathcal{T}$ in $B(x_{1},2\delta_{1})\setminus B(x_{1},\delta_{1})$, we conclude that the boundary of $\Omega_{1}:=\mathcal{T}^{-1}(D_{1})$ is $C^{1,\alpha}$ except at $x_1$, whereas $\widetilde \Omega_{1} :=\varphi_1(\Omega_{1})$ is $C^{1,\alpha}$ (for more details about localization and straightening, we refer to the proof of \cite[Theorem 3.9]{Pomm2}). Then, $f_{1}:= \varphi_1\circ \mathcal{T}^{-1}\circ g_{1}^{-1}$ is a Riemann mapping from $D$ to $\widetilde \Omega_{1}$. The Kellogg-Warschawski Theorem (see \cite[Theorem 3.6]{Pomm2}) now shows that $f_{1}\in C^1(\overline{D})$. Moreover, there exists $C_1>0$ such that
$$C_1^{-1}\leq |f_{1}'(\zeta)|\leq C_1\qquad \forall \zeta\in \overline{D}$$
(see \cite[Theorem 3.5]{Pomm2}). Similarly, we get the same properties for $g_{1}^{-1}$, so $\tilde f_{1}:= f_{1}\circ g_{1}\in C^1(D_{1})$ and 
$$\tilde C_1^{-1}\leq |\tilde f_{1}'(\zeta)|\leq \tilde C_1\qquad \forall \zeta\in \overline{D_{1}}$$
for some $\tilde C_1>0$.
The definition of $\tilde f_{1}$ immediately gives
\begin{equation}\label{holo}
 \frac{\theta_{1}}{\pi \tilde C_{1}} |\mathcal{T}^{-1}(y)-x_{1}|^{-\pi/\theta_{1} +1} \leq |(\mathcal{T}^{-1})'(y)|\leq \frac{\theta_{1}\tilde C_{1}}{\pi } |\mathcal{T}^{-1}(y)-x_{1}|^{-\pi/\theta_{1} +1} \qquad \forall y \in D_{1}
\end{equation}
and 
\begin{equation} \label{holo2}
 \frac{\pi }{\theta_{1}\tilde C_{1}} |x-x_{1}|^{\pi/\theta_{1} -1}\leq |\mathcal{T}'(x)|\leq \frac{\pi \tilde C_{1}}{\theta_{1}} |x-x_{1}|^{\pi/\theta_{1} -1} \qquad \forall x \in \Omega_{1}.
\end{equation}

By connectedness of $\Omega_{1}$, we know that for any $x \in \Omega_{1}$, there exists a smooth path $\gamma$ in $\Omega_{1}$ joining $x_{1}$ and $x$, and we have
\begin{equation*}
|\mathcal{T}(x)-\mathcal{T}(x_{1}) |
=\Big|\int_{0}^1 \mathcal{T}'(\gamma(t))\gamma'(t)d t \Big|
\leq \sup_{t\in (0,1)} |\mathcal{T}'(\gamma(t))| \ell(\gamma)\leq \frac{\pi \tilde C_{1}}{\theta_{1}} \ell(\gamma)^{\pi/{\theta_{1}}},
\end{equation*}
where we have used that $|\gamma(t)-x_{1}|\leq \ell(\gamma)$ and that $\frac\pi{\theta_{1}}-1>0$. We now claim that $\Omega_{1}$ is {\it a-quasiconvex} for some $a\geq 1$, that is, for any $x,y\in\Omega_1$ there exists a rectifiable path $\gamma$ joining $x,y$ and satisfying
$\ell(\gamma) \leq a |x-y|$.
This follows from \ref{H} because $\partial \Omega_{1}$ is a piecewise $C^1$ Jordan curve with no interior cusp and hence a quasidisc (see, e.g., \cite{Gustafsson}), and Ahlfors shows in \cite{Ahlfors} that in two dimensions we have
\[
\partial \Omega_{1} \text{ is a quasidisk} \Longleftrightarrow \Omega_{1} \text{ is quasiconvex}.
\]
It follows that there is $C_2>0$ such that
$$|\mathcal{T}(x)-\mathcal{T}(x_1)| \leq C_2|x-x_1|^{\pi/\theta_1} \qquad \forall x \in \Omega_{1},$$
so 
$$ |y-\mathcal{T}(x_1)|^{\theta_1/\pi} \leq C_2^{\theta_1/\pi} | \mathcal{T}^{-1} (y) -x_1| \qquad \forall y \in D_{1}.$$
If we also choose $\overline{D_1}$ convex (which can be done if we pick $\delta_1$ small enough), we also obtain
$$
|\mathcal{T}^{-1}(y)-x_1| = |\tilde f_{1}(y) -\tilde f_{1}(\mathcal{T}(x_1)) |^{\theta_1/\pi} \leq \tilde C_{1}^{\theta_1/\pi} |y-\mathcal{T}(x_1)|^{\theta_1/\pi}\qquad \forall y\in D_{1},
$$
which also implies
$$ |x-x_1|^{\pi/\theta_1} \leq \tilde C_{1} |\mathcal{T}(x) -\mathcal{T}(x_1)| \qquad \forall x\in \Omega_{1}. $$
These inequalities, together with \eqref{holo} and \eqref{holo2} yield the claims in the second bullet point for $k=1$, and similarly for any $k=2,\dots,N$. The claims in the first bullet point are obtained similarly, by considering a smooth domain $D_0\subseteq D$ such that 
\[
\Omega\setminus \bigcup_{k=1}^N B(x_{k},\delta_{k}) \subset \mathcal{T}^{-1}(D_{0}) \subset \Omega\setminus \bigcup_{k=1}^N B(x_{k},\delta_{k}/2),
\]
and using $\varphi_0(z):=z$.
\end{proof}

Putting together the above estimates on $\mathcal{T}'$ and $\mathcal{T}$, we obtain $C\ge 1$ such that for $k =1,\dots, N$ we have
\begin{equation}\label{est:DT1}
\begin{split}
&|D\mathcal{T}(\mathcal{T}^{-1}(y))|\leq C \qquad \forall y\in D, \\
&|D\mathcal{T}(\mathcal{T}^{-1}(y))|\leq C |y-\mathcal{T}(x_{k})|^{1-\frac{\theta_{k}}{\pi}} \qquad \forall y\in D\cap B(\mathcal{T}(x_{k}),\delta_0)
\end{split}
\end{equation}
when $\Omega$ satisfies \ref{H} and $\max_{k=1,\dots,N}\theta_{k} < \pi$.

\subsection{Particle trajectories and the approach of \cite{Lacave-SIAM}}\label{Sect.2.2}

Existence of global weak solutions to the 2D Euler equations in very general bounded domains (see Section~\ref{Sect.3.2} for the precise definition) was established in \cite{GV-Lac} in the Yudovich class \eqref{Yudo-class}.
With this level of regularity, local elliptic estimates allow us to define the Lagrangian flow up to the time of collision with the boundary. Indeed, following \cite[Chap. 2]{MarPul} we infer from $\omega\in L^\infty([0,\infty)\times \Omega)$ that $u$ is locally log-Lipschitz, and a classical extension of the Cauchy-Lipshitz theorem shows that for any $x\in \Omega$, there exists $t(x)>0$ and a unique curve $ X(\cdot,x)\in W^{1,\infty}([0,t(x)))$ such that $X(t,x)\in \Omega$ for each $t\in[0,t(x))$,
\[
X(t,x) = x + \int_{0}^t u(s,X(s,x))\, ds \qquad \forall t\in [0,t(x)),
\]
as well as $X(t(x),x)\in \partial\Omega$ if $t(x)<\infty$. As $u$ is uniformly (in time) log-Lipshitz on any compact subset of $\Omega$, we obtain
\[
\frac{d}{dt} X(t,x) = u(t,X(t,x)) \qquad \text{for a.e. } t\in [0,t(x)).
\]

If $u$ is not globally log-Lipshitz in $\Omega$ or does not belong to $\bigcap_{p\geq 2}W^{1,p}(\Omega)$, uniqueness of solutions is not known in general. However, it should hold if the vorticity is constant in the neighborhood of the set of points where the velocity is singular, which in our case are the obtuse corners (and possibly other points on the boundary with insufficient regularity). Here we will look for assumptions guaranteeing that the trajectories $X(\cdot,x)$ transporting $\omega$ that start inside $\Omega$ never reach $\partial\Omega$, which means, in particular, that if initially the vorticity is constant in a neighborhood of $\partial\Omega$, then it will remain such for all $t> 0$.

In order to illuminate our approach, we now recall the strategy of the proof of the main result from \cite{Lacave-SIAM}. Roughly speaking, the latter was inspired by the Lyapunov method developed by Marchioro \cite{marchioro}, but the Lyapunov function
\[
L(t):= -\ln|L_{1}(t,X(t,x))| \qquad \text{with} \qquad L_{1}(t,z):= \frac1{2\pi}\int_{\Omega} \ln\frac{| \mathcal{T}(z)- \mathcal{T}(y)|}{| \mathcal{T}(z)- \mathcal{T}(y)^*| | \mathcal{T}(y)|} \omega(t,y)\, dy
\] 
 is more complicated in \cite{Lacave-SIAM} because the singularity is weaker than in the case of a point vortex.
 As $L_{1}(t,\cdot)=\Delta^{-1}_z \omega(t,\cdot)$ (with the Dirichlet Laplacian), one can majorize $L_{1}$ by the distance to $\partial\Omega$. Hence one only needs to prove that $L$ stays finite in order to conclude that $t(x)=\infty$. We have 
\[
L'(t)=- \frac{\frac{d}{dt} X(t,x) \cdot \nabla_{z} L_{1}(t,X(t,x)) +\partial_{t} L_{1}(t,X(t,x)) }{L_{1}(t,X(t,x))} = - \frac{\partial_{t} L_{1}(t,X(t,x)) }{L_{1}(t,X(t,x))},
\]
where the (most singular) first term vanishes due to $u(t,z)=\nabla^\perp_{z} L_{1}(t,z)$ (this motivated the choice of $L$ in \cite{marchioro,Lacave-SIAM}).
To conclude, one needs to estimate $\partial_{t} L_{1}(t,z)$ by $L_{1}(t,z)$, particularly where $L_{1}(t,z)=0$. When $\omega$ has a definite sign, $L_1$ only vanishes on $\partial\Omega$, and a technical lemma is used in \cite{Lacave-SIAM} to prove that $\partial_{t} L_{1}$ also vanishes on $\partial\Omega$ and to control the relevant rate by $L_{1}$. The sign condition is needed for this choice of Lyapunov function $L$, as otherwise $L_{1}$ can vanish inside $\Omega$. The main motivation for this Lyapunov function was to treat large angles like the interior cusp, and it turned out that the arguments worked as long as all $\theta_{k}\in (\frac \pi 2,2\pi]$. For a more detailed discussion of the Lyapunov method in other contexts (such as the vortex-wave system in $\R^2$ or Euler equations with fixed point vortices in $\R^2$), we refer to \cite[Sect. 7.5]{Lacave-SIAM}.

In contrast to \cite{Lacave-SIAM}, we use here a much simpler Lyapunov function to obtain Theorem~\ref{main1}, which allows us to discard the sign condition on $\omega$ when all corners of $\Omega$ have angles smaller than $\pi$. Conversely, in the proof of our Theorem~\ref{main2} we show that the sign condition is in fact necessary to prevent particle trajectories reaching $\partial\Omega$ in finite time for general domains with corners whose angles are greater than $\pi$.

\section{Control of trajectories for domains with convex corners}\label{sec.3}
\subsection{The Lyapunov function and the proof of Theorem~\ref{main1}(i)} \label{sect.Lyapunov}

Let us consider a global weak solution $(u,\omega)$ of the Euler equations, a point $x\in \Omega$ and the trajectory $X(\cdot,x)$ starting at $x$. As $\mathcal{T}$ maps $\Omega$ to $D$ and $\partial \Omega$ to $\partial D$, it is clear that for any $t<t(x)$
\[
L(t):=1 - \ln \Big(1- | \mathcal{T}(X(t,x)) | \Big) \in [1, \infty).
\]
If $t(x)<\infty$, then we must have $\lim_{t\to t(x)} L(t)=\infty$. That is, to prove that the trajectory does not reach the boundary in finite time, we need to show that $L$ stays bounded on bounded intervals. We have
\begin{equation*}
 L'(t)= \frac{ \mathcal{T}(X(t,x))\cdot \left(D\mathcal{T}(X(t,x)) \frac{d}{dt} X(t,x)\right) }{| \mathcal{T}(X(t,x))| (1- | \mathcal{T}(X(t,x)) |)}
\end{equation*}
whenever $|\mathcal{T}(X(t,x))|\in(0,1)$.
As $\mathcal{T}$ is holomorphic, $D\mathcal{T}$ is of the form $\begin{pmatrix} a & b \\ -b & a \end{pmatrix}$ and we have $D\mathcal{T}D\mathcal{T}^T=(\det D\mathcal{T}) \rm{I}_{2}$. Using the Biot-Savart formula \eqref{Biot-Savart} to evaluate $\frac{d}{dt} X(t,x)$ now yields
\begin{align*}
 L'(t)= &
 \frac{ \det D\mathcal{T}(X(t,x)) }{2\pi | \mathcal{T}(X(t,x))| (1- | \mathcal{T}(X(t,x)) |)} \int_{\Omega} \left( \frac{ - \mathcal{T}(X(t,x)) \cdot \mathcal{T}(y)^\perp}{| \mathcal{T}(X(t,x)) - \mathcal{T}(y)|^2} + \frac{ \mathcal{T}(X(t,x)) \cdot \mathcal{T}(y)^{*\perp}}{| \mathcal{T}(X(t,x))- \mathcal{T}(y)^*|^2 } \right) \omega(t,y)\, dy \\
 =& \frac{ \det D\mathcal{T}(X(t,x)) (1 + | \mathcal{T}(X(t,x)) |)}{2\pi | \mathcal{T}(X(t,x))| } \int_{\Omega} \frac{|\mathcal{T}(y)|^2 (|\mathcal{T}(y)|^2-1) \mathcal{T}(X(t,x)) \cdot \mathcal{T}(y)^\perp }{| \mathcal{T}(X(t,x)) - \mathcal{T}(y)|^2| |\mathcal{T}(y)|^2 \mathcal{T}(X(t,x))- \mathcal{T}(y)|^2 } \omega(t,y)\, dy,
\end{align*}
where we have used $z^*=z|z|^{-2}$. This implies that
\[
 L'(t) \leq \frac{ 2 \| \omega \|_{L^\infty} \det D\mathcal{T}(X(t,x)) }{\pi | \mathcal{T}(X(t,x))| } \int_{\Omega} \frac{ (1- |\mathcal{T}(y)|) |\mathcal{T}(X(t,x)) \cdot \mathcal{T}(y)^\perp | }{| \mathcal{T}(X(t,x)) - \mathcal{T}(y)|^2| |\mathcal{T}(y)|^2 \mathcal{T}(X(t,x))- \mathcal{T}(y)|^2 }\, dy.
\]
Theorem~\ref{main1}(i) now follows from the following technical lemma.

\begin{lemma}\label{lem-technic}
 Let $\Omega\subseteq\mathbb R^2$ be a bounded open domain satisfying \ref{H}, with $\max_{k} \theta_{k}<\pi$. Then there is $C_\Omega>0$ such that
 \beq \label{3.1}
 \det D\mathcal{T}(\mathcal{T}^{-1}(\xi) ) \int_{D} \frac{ (1- |z|) |\xi \cdot z^\perp | }{| \xi - z|^2| |z|^2 \xi- z|^2 }\det D \mathcal{T}^{-1} (z) \, dz \leq C_\Omega \left| \ln (1-| \xi|)\right| 
 \eeq
for all $\xi\in D\setminus B(0,\frac 12 )$.
\end{lemma}

The proof is postponed to Appendix~\ref{app-lemma}. The lemma and the change of variables $z=\mathcal{T}(y)$ show that
\[
L'(t)\leq C \| \omega \|_{L^\infty}\left| \ln (1-| \mathcal{T}(X(t,x)) |)\right| \le C\| \omega \|_{L^\infty} L(t)
\] 
when $L(t)\ge 2$ (because then $|\mathcal{T}(X(t,x))|>\frac 12$), with $C$ depending on $D$. This and $L(t)\ge 1$ now imply $L(t)\leq (1+L(0)) e^{C\| \omega \|_{L^\infty}t}\le 2L(0) e^{C\| \omega \|_{L^\infty}t}$, so
\[
 | \mathcal{T}(X(t,x)) | \leq 1- \exp(-2L(0)e^{C\| \omega \|_{L^\infty}t}).
\]
Hence the trajectory $X(\cdot,x)$ cannot reach $\partial\Omega$ in finite time, and Theorem~\ref{main1}(i) is proved.

\subsection{Renormalized solutions and the proofs of Theorem~\ref{main1}(ii,iii)}\label{Sect.3.2}

The goal of this section is to show how Theorem~\ref{main1}(i) implies Theorem~\ref{main1}(ii,iii).

Consider an initial vector field $u_{0}\in L^2(\Omega)$ such that $\curl u_{0}\in L^\infty(\Omega)$ and $u_0$ verifies the divergence free and impermeability conditions in the weak sense, which in bounded domains means 
\begin{equation}\label{imperm}
 \int_\Omega u_0 \cdot h \,dx = 0 \qquad \forall h \in 
 %G(\Omega):=\{w\in L^2(\Omega) \ : \ w=\nabla p \ \text{ for some } p\in H^1_{\loc}(\Omega)\}.
 G(\Omega):=\{\nabla p   \ : \  p\in H^1(\Omega)\}.
 \end{equation}
 Similarly, we can instead consider an initial function $\omega_{0}\in L^\infty(\Omega)$, and then there exists a unique $u_{0}\in L^2(\Omega)$ with $\curl u_0=\omega_0$ that verifies \eqref{imperm}.
 
We say that $u$ belonging to the Yudovich class \eqref{Yudo-class} is a weak solution of the velocity formulation \eqref{eq.Euler1}-\eqref{eq.Euler3} with initial data $u_{0}$ whenever
\begin{equation} \label{Eulerweak}
 \int_0^{\infty} \int_\Omega \left( u \cdot \partial_t \varphi + (u \otimes u) : \nabla \varphi \right) dxdt = -\int_\Omega u_0 \cdot \varphi(0, \cdot) dx \qquad \forall \, \varphi \in \mathcal{D}\left([0, +\infty) \times \Omega\right) \text{ with } \div \varphi = 0
\end{equation}
and
\begin{equation} \label{imperm2}
\int_{\R^+} \int_\Omega u \cdot h \,dx dt= 0 \qquad \forall h \in \mathcal{D}\left([0,+\infty); G(\Omega)\right).
\end{equation}
This is equivalent to having a pair $(u,\omega)$ with $u$ in the Yudovich class \eqref{Yudo-class} and $\omega=\curl u$ in $\mathcal{D}'((0,+\infty)\times \Omega)$ that is a weak solution of the vorticity formulation \eqref{eq.Euler2}-\eqref{eq.Euler5} for initial data $\omega_{0}$ in the sense of \eqref{imperm2} and
\begin{equation} \label{Vortweak}
\int_0^{\infty} \int_\Omega \left( \omega \partial_t \varphi + \omega u \cdot \nabla \varphi \right) \,dx dt= -\int_\Omega \omega_0 \varphi(0, \cdot) dx \qquad \forall \, \varphi \in \mathcal{D}\left([0, +\infty) \times \Omega\right).
\end{equation}

Without any assumption on the regularity of $\partial\Omega$, existence of a global weak solution was established in \cite{GV-Lac} (see \cite[Remark 1.2]{GV-Lac2} for the vorticity formulation).
Let us consider such a solution when $\Omega$ is a domain verifying \ref{H}. As explained in Section~\ref{sec.2}, we can write $u$ in terms of $\omega:=\curl u$ through the Biot-Savart law \eqref{Biot-Savart} and construct for every $x\in \Omega$ a $W^{1,\infty}$-in-time trajectory $X(\cdot,x)$ starting at $x$. The main result of the previous section is that these trajectories never reach $\partial\Omega$ in finite time and they are then defined for every $t\in [0,\infty)$.

Nevertheless, it is not obvious for weak solutions that the vorticity is transported by the flow, namely that $\omega(t,X(t,x))=\omega_{0}(x)$. To get this property, we recall in Step 1 below that the solution is more regular than \eqref{Yudo-class} when $\Omega$ satisfies \ref{H} with $\alpha=1$, and that it is actually a renormalized solution in the sense of DiPerna-Lions. This will imply that $\omega(t,X(t,x))=\omega_{0}(x)$. We can then also conclude from the previous section that the vorticity stays constant in the neighborhood of the boundary if it verifies this property initially, which then implies uniqueness of global weak solutions as we show in Step 2. 

\bigskip

\noindent {\bf Step 1: Renormalized solutions}

When the domain is piecewise $C^{1,1}$ with the corners having arbitrary angles, and the vorticity is bounded, it is known from elliptic theory in domains with corners (see, e.g., \cite{Kondra,Grisvard, Mazya}) that the stream function $\psi:= \Delta^{-1}\omega$ (with $\Delta$ the Dirichlet Laplacian) belongs to $W^{2,p}(\Omega)$ for any $p\in [1,\frac 43)$. (We recall that this is false in general for $C^1$ domains, see \cite{Kenig}.) Then $\omega$ solves (in the weak sense of \eqref{Vortweak}) the transport equation with the advecting vector field $u=\nabla^\perp \psi$ and
\begin{equation*}
\psi \in L^\infty([0,\infty);W^{2,5/4}(\Omega)). 
\end{equation*}
This regularity would allow us to apply DiPerna-Lions theory for linear transport equations \cite{DiPernaLions}, but the latter was developed only for smooth domains. To bypass this restriction, we extend $(u,\omega)$ on $\R^2$ as follows. First, we note that $\Omega$ verifies the {\it Uniform Cone Condition} (see \cite[Par. 4.8]{Adams} for the precise definition) because $\theta_{k}>0$ for all $k=1,\dots,N$. Therefore \cite[Theorem 5.28]{Adams} states that there exists a {simple} $(2,\frac 54)$-extension operator $E:W^{2,5/4}(\Omega)\to W^{2,5/4}(\R^2)$, that is, there exists $K>0$ such that for any $v \in W^{2,5/4}(\Omega)$
\[
Ev=v \text{ a.e. in } \Omega \qquad\text{and} \qquad \| Ev \|_{W^{2,5/4}(\R^2)} \leq K \| v \|_{W^{2,5/4}(\Omega)}.
\]
Introducing a smooth cutoff function $\chi$ such that $\chi \equiv 1$ on $B(0,R)$ and $\chi \equiv 0$ on $B(0,R+1)$, with $R$ large enough so that $\Omega\subset B(0,R)$, we let for a.e. $t\ge 0$
\[
\bar \psi(t,\cdot) = \chi E\psi(t,\cdot) \qquad \text{and} \qquad \bar u(t,\cdot) = \nabla^{\perp} \bar \psi(t,\cdot).
\]
Hence we have for a.e. $t\ge 0$,
\begin{equation}\label{3.11}
\bar u(t,\cdot) =u(t,\cdot) \text{ a.e. on }\Omega,
\end{equation}
and
\begin{equation*}
\div \bar u(t,\cdot)=0 \text{ a.e. on }\R^2 \qquad \text{and} \qquad
 \bar u \in L^\infty(\R_+;W^{1,5/4}(\R^2)).
\end{equation*}

We next let $\bar\omega$ be the extension of $\omega$ by zero outside $\Omega$ and we note that $\bar \omega$ is a weak solution (see \eqref{Vortweak}, with $\mathbb R^2$ in place of $\Omega$) of the transport equation
\[
 \partial_{t} \bar \omega + \bar u \cdot \nabla \bar \omega=0 \qquad \text{and} \qquad \bar \omega(0,\cdot)=\overline{\omega_0}.
\]
To prove this, one only needs to consider test functions in \eqref{Vortweak} whose the support intersects $[0,\infty)\times \partial\Omega$. This was done in \cite[Lemma 4.3]{LMW} for angles less than or equal to $\frac \pi 2$ using log-Lipschitz regularity of $u$ close to acute corners, and in \cite[Proposition 2.5]{Lacave-SIAM} for angles in $(\frac\pi 2,2\pi]$ using tangency properties hidden in the explicit form of the Biot-Savart law \eqref{Biot-Savart} (see \cite[Lemma 2.6]{Lacave-SIAM}). By using appropriate cutoff functions supported near the corners, one can use these two results to obtain the desired claim about $\bar \omega$.

Therefore, the results of DiPerna and Lions \cite{DiPernaLions} on linear transport equations ensure that $\bar \omega$ is the unique weak solution in $L^\infty([0,\infty), L^{5}(\R^2))$ to the linear transport equation with velocity field $\bar u$. For a precise statement, we refer to \cite[Theorem II.2]{DiPernaLions}; we also refer to, e.g., \cite[Section 4]{Ambrosio} for more recent developments in the theory. 

Next, Theorem~\ref{main1}(i) shows that $X(t,\Omega)\subseteq \Omega$ for all $t\ge 0$. Using this and \eqref{3.11}, one can readily prove that $\tilde{\omega}(t):=X(t,\cdot)_\#\bar \omega_0$ is an $L^\infty([0,\infty),L^{5}(\R^2))$ (recall that $\omega_0$ is bounded) solution to the same transport equation but with velocity field $\bar u$ (see e.g. the proof of Proposition 2.1 in \cite{Ambrosio}). As $\bar \omega_{0}\equiv 0$ in $\Omega^c$, we can consider any extension of $X$ on $\R^2$, for instance the flow map associated to $\bar u$. Due to uniqueness, we can now conclude that $\bar \omega(t)=\tilde{\omega}(t)$ for a.e. $t\ge 0$, which in particular yields
 $$\omega(t)=X(t,\cdot)_\#\omega_0, \quad \text{for a.e. }t\ge 0 $$
in the sense that for a.e. $t\ge 0$ we have $\int_\Omega \omega(t,x)\varphi(x)\,dx=\int_\Omega \omega_0(x)\varphi(X(t,x))\,dx$ for all $\varphi\in C_c(\Omega)$.

After redefining $\omega$ on a set of measure zero, this becomes $\omega(t,x)=\omega_{0}(X^{-1}(t,x))$. Uniform boundedness of $u$ on any compact subset of $\Omega$, which follows from boundedness of $\omega$, now yields $\omega\in C([0,\infty);L^1(\Omega))$. It is then not hard to show, using the Biot-Savart law, that $u$ is continuous on $[0,\infty)\times \Omega$, which also means that \eqref{1.11} holds for all $(t,x)\in[0,\infty)\times\Omega$. We therefore proved Theorem~\ref{main1}(ii).

\bigskip

\noindent {\bf Step 2: Uniqueness if the vorticity is constant in the neighborhood of the singular part of $\partial\Omega$.}

We prove here a stronger result than Theorem~\ref{main1}(iii). Namely, that appropriate solutions to the 2D Euler equations in the sense of Theorem~\ref{main1}(ii) are unique on {\it more general domains} $\Omega$, as long as the vorticity is constant in a neighborhood of the part of $\partial\Omega$ where $\partial\Omega\notin C^{2,\tilde\alpha}$. Let $\tilde\alpha>0$ be arbitrary and denote by 
\[
\Gamma_{\tilde\alpha}:= \{ x \in \partial \Omega\,:\, \partial \Omega \cap B(x,\varepsilon) \notin C^{2,\tilde \alpha} \text{ for all }\varepsilon>0\}
\]
the singular part of $\partial \Omega$. In particular, all corners of $\Omega$ belong to this set.

\begin{proposition}\label{constant_vorticity} 
Let $\Omega\subseteq\bbR^2$ be an open bounded simply connected domain and let $u$ be a global weak solution to the Euler equations on $\Omega$ from the Yudovich class \eqref{Yudo-class} such that $\omega(t,\cdot)=\omega(0,X^{-1}(t,\cdot))$ for all $t>0$.
If there is $a\in\R$ such that $\supp(\omega(0,\cdot)-a)\cap\Gamma_{\tilde\alpha}=\emptyset$ (for some $\tilde\alpha>0$), then $u$ is the unique such solution with initial value $\omega(0,\cdot)$ until the first time $t$ such that $\supp (\omega(t,\cdot)-a)\cap \Gamma_{\tilde\alpha}\neq \emptyset$. 
\end{proposition}

Theorem~\ref{main1}(iii) will then follow from this and from Theorem~\ref{main1}(i,ii), because Theorem~\ref{main1}(i) shows that the (closed) support of $\omega-a$ can never reach $\partial\Omega$. We prove Proposition~\ref{constant_vorticity} by adapting the uniqueness proof of Marchioro and Pulvirenti \cite{MarPul} to non-smooth domains.

\begin{proof}[Proof of Proposition~\ref{constant_vorticity}]
Let $\omega_0:=\omega(0,\cdot)$ and assume without loss that $\|\omega_0\|_{L^\infty}\le 1$. Let $T>0$ be any time by which the support of $\omega-a$ did not reach $\Gamma_{\tilde\alpha}$. Let $\Omega_0\subseteq \Omega$ be an open set such that $|\partial\Omega_0|=0$ and $\omega_0\equiv a$ on $\Omega\setminus\Omega_0$, as well as $\overline{\bigcup_{t\in[0,T)} X(t,\Omega_0)}$ (which is compact) contains no point from $\Gamma_{\tilde\alpha}$.
 Let $C_T>0$ and an open set $\Omega_T\subseteq\Omega$ containing $\Omega\cap \overline{\bigcup_{t\in[0,T)} X(t,\Omega_0)}$ be such that with $d_T(\cdot,\cdot)$ the distance function in $\Omega_T$ we have $d_T(x,y)\le C_T|x-y|$ whenever $x,y\in\Omega_T$ and $|x-y|\le C_T^{-1}$, as well as
\[
|D^j_{x} G_\Omega(x,y)| \le \frac {C_T}{|x-y|^j} \qquad\text{for $j=1,2$}
\]
for each $(x,y)\in\Omega_T\times\Omega$. Here $G_\Omega$ is the Dirichlet Green's function for $\Omega$ and existence of $\Omega_T$ follows from the definition of $\Omega_0$ and the relation $G_\Omega(x,y)=G_{D}( \mathcal{T}(x), \mathcal{T}(y))$. Indeed, the Kellogg-Warschawski Theorem shows that $\mathcal{T}$ is $C^2$ away from $\Gamma_{\tilde\alpha}$, so the above bounds on $G_\Omega$ for $x$ away from $\Gamma_{\tilde\alpha}$ follow from the same bounds on the explicitly given function $G_D$.

Let $Y$ be the flow map of another solution $w(t,\cdot)=\omega(0,Y^{-1}(t,\cdot))$ as above, 
%and satisfying $Y(0,\cdot)={\rm Id}$, 
and let
\[
\eta(t):=|\Omega_0|^{-1} \int_{\Omega_0} | X(t,x)- Y(t,x)| dx.
\]
Let $T'\le T$ be the latest time such that $\Omega\cap \overline{\bigcup_{t\in[0,T'')} Y(t,\Omega_0)}\subseteq \Omega_T$ for any $T''\in[0,T')$. Let $K_\Omega:=\nabla^\perp_x G_\Omega$, so that the Biot-Savart laws \eqref{Biot-Savart} for $\omega$ and $w$ read
\[
u(t,x) :=\int_\Omega K_\Omega(x,y)\omega(t,y)dy\qquad\text{and}\qquad
 v(t,x) :=\int_\Omega K_\Omega(x,y) w(t,y)dy.
\]
Let 
\[
\phi(r):=
\begin{cases}
r(1-\ln r) & r\in(0,1),
\\ 1 & r\ge 1.
\end{cases}
\]
A standard argument using the above bounds on $G_\Omega$ (see Appendix 2.3 in \cite{MarPul}) shows that
\[
\max \left\{\int_{\Omega} |K_\Omega(x,y) - K_\Omega(x',y)| dy , \int_{\Omega} |K_\Omega(y,x) - K_\Omega(y,x')| dy \right\} \le C \phi(|x-x'|)
\]
for all $x,x'\in\Omega_T$ and some $C$ depending only on $C_T$, $\Omega_{T}$, and $\Omega$. This then also implies 
\begin{equation}\label {3.20}
|u(t,x)-u(t,x')| \le C \phi(|x-x'|)
\end{equation}
 for any $t< T'$ and $x,x'\in\Omega_T$ (recall that $\|\omega\|_{L^\infty}\le 1$). 
We have $ |\Omega_0|^{-1} \int_{\Omega_0} \phi(f(x))dx\le \phi(|\Omega_0|^{-1} \int_{\Omega_0} f(x)dx)$ for any $f$, due to concavity of $\phi$ and Jensen's inequality. Hence for any $t<T'$ we have
\begin{align*}
\eta(t) &\le |\Omega_0|^{-1} \int_{\Omega_0} \int_0^t |u(s,X(s,x))-u(s,Y(s,x))| ds dx + |\Omega_0|^{-1}  \int_{\Omega_0} \int_0^t |u(s,Y(s,x))-v(s,Y(s,x))| ds dx,
\\ & \le C\int_0^t \phi(\eta(s)) ds + |\Omega_0|^{-1} \int_0^t \int_{\Omega_0} \left|\int_\Omega K_\Omega (Y(s,x),y)\omega(s,y)dy - \int_\Omega K_\Omega (Y(s,x),y)w(s,y)dy \right| dx ds
\\ & = C\int_0^t \phi(\eta(s)) ds + |\Omega_0|^{-1} \int_0^t \int_{\Omega_0} \left|\int_\Omega \big[K_\Omega (Y(s,x),X(s,y))- K_\Omega (Y(s,x),Y(s,y)) \big] \omega_0(y)dy \right| dx ds,
\end{align*}
where at the end we used the measure-preserving changes of variables $y\mapsto X(s,y)$ and $y\mapsto Y(s,y)$. 

Since $\omega_0\equiv a$ on $\Omega\setminus \Omega_0$, we get
\begin{align*}
\int_{\Omega\setminus \Omega_0} K_\Omega (z,X(s,y))\omega_0(y) dy &= a \int_\Omega K_\Omega (z,X(s,y)) dy - a \int_{\Omega_0} K_\Omega (z,X(s,y)) dy 
\\ & = a \int_\Omega K_\Omega (z,y) dy - a \int_{\Omega_0} K_\Omega (z,X(s,y)) dy 
\end{align*}
for any $z\in\Omega$. Similarly
\[
\int_{\Omega\setminus \Omega_0} K_\Omega (z,Y(s,y))\omega_0(y) dy = a \int_\Omega K_\Omega (z,y) dy - a \int_{\Omega_0} K_\Omega (z,Y(s,y)) dy, 
\]
so this, $|a|\leq \|\omega_0\|_{L^\infty}\le 1$, and the measure-preserving change of variables $Y(s,x)\mapsto x$ yield 
\begin{align*}
\eta(t) & \le C\int_0^t \phi(\eta(s)) ds + 2 |\Omega_0|^{-1} \int_0^t \int_{\Omega_0}\int_{\Omega_0} \big|K_\Omega (Y(s,x),X(s,y))- K_\Omega (Y(s,x),Y(s,y)) \big|dy dx ds
\\ & \le C\int_0^t \phi(\eta(s)) ds + 2 |\Omega_0|^{-1} \int_0^t \int_{\Omega_0}\int_{\Omega} \big|K_\Omega (x,X(s,y))- K_\Omega (x,Y(s,y)) \big|dx dy ds
\\ & \le C\int_0^t \phi(\eta(s)) ds + 2C |\Omega_0|^{-1} \int_0^t \int_{\Omega_0}\phi(|X(s,y)-Y(s,y)|) dy ds
\\ & \le 3C\int_0^t \phi(\eta(s)) ds
\end{align*}
for any $t<T'$. Since $\eta(0)=0$, it follows that $\eta\equiv 0$ on $[0,T')$. So $X(t,\cdot)|_{\Omega_0}=Y(t,\cdot)|_{\Omega_0}$ for all $t\in[0,T')$, which means that $\omega\equiv w$ on $[0,T')\times\Omega$. Therefore also $T'=T$, by the definition of $T'$, finishing the proof.
\end{proof}

\section{Reaching the boundary in finite time at concave corners}\label{sec.4}

We now prove Theorem~\ref{main2}. Our domain $\Omega$ will be any domain satisfying the hypotheses which is also symmetric across the $x_1$ axis and its intersection with $D$ is the set of all $z\in D\setminus\{0\}$ with ${\rm arg}(z)\in (-\frac\theta 2,\frac\theta 2)$ (while in the case $\theta=2\pi$ the domain has an inward cusp at the origin). We let $\omega(0,\cdot)\not\equiv 0$ be supported inside $\Omega$, odd in $x_2$ and such that $\sgn(x_2)\omega(0,x)\le 0$. Steps 1 and 2 of Section~\ref{Sect.3.2} (in particular, uniqueness) and the odd symmetry show that these properties continue to hold for $\omega(t,\cdot)$ as long as $\supp \omega(t,\cdot)=X(t,\supp \omega(0,\cdot))$ does not reach $\partial\Omega$.
We let $\Omega^+:=\Omega\cap(\bbR\times(0,\infty))$, $\Omega^-:=\Omega\cap(\bbR\times(-\infty,0))$, and $\Omega^0:=\Omega\cap(\bbR\times\{0\})$

By the Riemann mapping theorem, there is a unique biholomorphism $\mathcal{T}:\Omega\to D$ such that $\mathcal{T}((\frac12,0))=0$ and $\mathcal{T}'((\frac12,0))\in (0,\infty)$. It is easy to show that $\widetilde{\mathcal{T}}:=R\circ \mathcal{T}\circ R$, with $R(x_1,x_2):=(x_1,-x_2)$, has the same properties, hence
\[
 \mathcal{T}=R\circ \mathcal{T}\circ R.
\]
So we obtain $G_{\Omega^\pm}(x,y)=G_\Omega(x,y)-G_\Omega(x,Ry)$ for $x,y\in\Omega^\pm$, with 
\[
G_\Omega(x,y)=\frac 1{2\pi}\ln\frac{|\calT(x)-\calT(y)|}{|\calT(x)-\calT(y)^*| |\calT(y)|}
\]
the Green's function on $\Omega$, where we recall that the Green's functions satisfy
\[
G_{\Omega^\pm}(x,y)=G_{\Omega^\pm}(y,x) \ \forall (x,y)\in (\Omega^\pm)^2, \quad G_{\Omega^\pm}(x,y)=0 \ \forall (x,y)\in \Omega^\pm\times \partial\Omega^\pm, \quad \Delta_{x} G_{\Omega^\pm}(\cdot,y) = \delta(\cdot-y) \ \forall y \in \Omega^\pm.
\]
As the Biot-Savart law \eqref{Biot-Savart} can be written as
\[
u(t,x)= \int_\Omega \nabla_x^\perp G_\Omega(x,y) \omega(t,y)dy,
\]
it follows that if any solution on $\Omega^\pm$ is extended onto $\Omega$ via an odd-in-$x_2$ reflection (that is, $\omega(t,Rx)=-\omega(t,x)$ and $u(t,Rx)=Ru(t,x)$), this extension will be a solution on $\Omega$. And conversely, the restriction to $\Omega^\pm$ of any odd-in-$x_2$ solution on $\Omega$ is also a solution on $\Omega^\pm$.

If $\omega(0,\cdot)$ is as above and $\supp \omega(0,\cdot)\subseteq\Omega^+\cup\Omega^-$, then Theorem~\ref{main1} shows that we have a unique global solution on $\Omega^+$ with the initial vorticity being restricted to that set, and its trajectories starting inside $\Omega^+$ never reach $\partial\Omega^+$. Its odd reflection on $\Omega$ is then a global solution $\omega$ on $\Omega$ whose trajectories starting in $\Omega^+\cup\Omega^-$ never reach $\partial\Omega\cup\Omega^0$. Hence the solution is unique, and $\omega(t,\cdot)$ vanishes on a neighborhood of $\partial\Omega\cup\Omega^0$ for each $t\ge 0$. We will show below that, nevertheless, the trajectory $X(t,x)$ for any $x\in\Omega^0$ does reach the origin (and hence $\partial\Omega$) in finite time.

If $\Gamma:=\supp \omega(0,\cdot)\cap\Omega^0\neq \emptyset$, the solution remains unique as long as no trajectory starting at some $x\in \Gamma$ reaches the origin. This is because the restriction of the solution to $\Omega^\pm$ is a solution on that set, and hence trajectories starting in $\Omega^\pm$ cannot reach $\partial \Omega^\pm$ in finite time by Theorem~\ref{main1} (while trajectories starting in $\Omega^0\setminus \Gamma$ do not pose a problem). We will show below that $X(t,\Gamma)$ does reach the origin in some (first) time $t_0>0$, and does not reach the other end of $\Omega^0$ before time $t_0$ because our choice of $\omega(0,\cdot)$ ensures that $u_2(t,x)=0$ and $u_1(t,x)<0$ for each $(t,x)\in[0,t_0)\times\Omega^0$ (see below). Therefore the solution remains unique up to time $t_0$ and $0\in\supp \omega(t_0,\cdot)\cap\partial\Omega$.

The claims in the last two paragraphs (and hence Theorem~\ref{main2}) will be proved once we show existence of $\nu<1$ such that for each $\beta<1$ there is $C_\beta>0$ such that 
\begin{equation}\label{u1}
u_1(t,x)\le -C_\beta x_1^\nu 
\end{equation}
for all $(t,x)\in[0,t_0)\times\beta\Omega^0$, with $t_0$ the first time such that $0\in X(t_0,\Gamma)$. Because of the symmetry of $\omega$, it is sufficient to show this for the solution on $\Omega^+$, which satisfies $\omega\le 0$ and is a restriction to $\Omega^+$ of the solution from $\Omega$.

\smallskip
\noindent{\it Remark.}
 As $\theta>\pi$, the velocity $u$ is in general unbounded close to the corner. Nevertheless, for odd-in-$x_2$ vorticities, $u$ is given by the Biot-Savart law on $\Omega^{+}$, where the corner has angle $\frac \theta 2\leq \pi$. So the symmetry cancels the most singular term in the Biot-Savart law on $\Omega$, and $u$ will be bounded and continuous on $\overline\Omega$. If $\theta<2\pi$, this continuity and the tangency condition imply that $u$ vanishes at the corner, so $\theta >\pi$ (i.e., $\frac \theta 2 > \frac\pi 2$) is necessary to have \eqref{u1} with $\nu<1$.
\smallskip

Let us now show \eqref{u1}.
Let $\calT:\Omega^+\to D$ be as before (but for $\Omega^+$ and mapping $(\frac 12,\frac 12)$ to $0$). Since 
\[
\frac{|\zeta-z|^2}{|\zeta-z^*|^2|z|^2}=1-\frac{(1-|\zeta|^2)(1-|z|^2)}{|\zeta-z^*|^2|z|^2}\in [0,1] 
\]
for $z\in D$ and $\zeta\in \overline D$, we have $G_D(\zeta,z)< 0$ when $\zeta\in D$ and $G_D(\zeta,z)=0$ when $\zeta\in\partial D$ (recall that $G_D(\zeta,z)=\frac 1{2\pi}\ln\frac{|\zeta-z|}{|\zeta-z^*| |z|}$). Hence when $x\in\partial\Omega^+$ (so that $|\calT(x)|=1$), the integrand in \eqref{Biot-Savart} is 
\[
 2\pi \nabla_{\zeta}^\perp G_D(\calT(x),\calT(y)) \omega(t,y) = 
 2\pi | \nabla_{\zeta} G_D(\calT(x),\calT(y))| \calT(x)^\perp\omega(t,y).
 \] 
 Moreover, from Step 1 in Section~\ref{Sect.3.2} we know that the vorticity is transported by the flow $X$, so the $L^\infty(\Omega^+)$ and $L^1(\Omega^+)$ norms of $\omega$ are conserved. So for $\gamma\in (0,1)$ close enough to $1$ and all $t\ge 0$ we have
\[
\| \omega(t,\cdot )\|_{L^1(\mathcal{T}^{-1}(B(0,\gamma)))} \geq \frac 12 \|\omega_{0}\|_{L^1(\Omega^+)}.
\]
This, $\omega\le 0$ on $\Omega^+$, and
\[
\inf_{(\zeta,z)\in\partial D\times\gamma D} | \nabla_{\zeta} G_D(\zeta,z)|=\inf_{(\zeta,z)\in\partial D\times\gamma D} \left|\frac{\zeta-z}{|\zeta-z|^2} - \frac{\zeta-z^*}{|\zeta-z^*|^2}\right|>0
\] 
show that the integral in \eqref{Biot-Savart} is a vector $-g(t,x)\calT(x)^\perp$ with
\[
\inf_{(t,x)\in[0,t_0)\times \partial\Omega^+}g(t,x)>0.
\] 
It therefore suffices to analyze the time-independent term $-D\calT^T(x)\calT(x)^\perp$ for $x\in\Omega^0$. Since $D\mathcal{T}$ is of the form $\begin{pmatrix} a & b \\ -b & a \end{pmatrix}$, we have $D\mathcal{T}^T D\mathcal{T}=(\det D\mathcal{T}) \rm{I}_{2} = |\nabla | \mathcal{T}(x)||^2 I_2$, with the last equality due to $x\in \partial \Omega^+$. Hence
\[
-D\mathcal{T}^T(x) \calT(x)^\perp = -D\calT^T(x)|\nabla | \mathcal{T}(x)||^{-1}D\calT(x)\tau_x = - \sqrt{\det D\mathcal{T}(x)} \tau_{x} ,
\]
with $\tau_x$ the counter-clockwise unit tangent to $\partial\Omega^+$ at $x$. Proposition~\ref{prop T} shows that for each $\beta<1$ there is $C_\beta>0$ such that $\sqrt{\det D\mathcal{T}(x)}= \sqrt{a^2+b^2}\geq \frac 12 |D \calT(x)| \geq C_\beta |x|^{2\pi/\theta -1}$ for all $x\in\beta\Omega^0$, which 
proves \eqref{u1} with $\nu:=\frac{2\pi} \theta -1$. The proof is finished. 

\appendix
\section{Proof of Lemma~\ref{lem-technic}} \label{app-lemma}

Let $\delta \in (0,\delta_{0}]$ be such that $\calT(B(x_{k},\delta))\subseteq B(\mathcal{T}(x_{k}),\delta)$ for all $k=1,\dots,N$, with $\delta_{0}>0$ from Proposition~\ref{prop T}. 
Such $\delta$ exists because $\max_k \theta_k<\pi$.

Let us first assume that $\xi$ is not near any $\calT(x_k)$ (with $x_k$ the corners of $\Omega$). Specifically, we assume that $\calT^{-1}(\xi)\in \Omega\setminus \bigcup_{k=1}^N B(x_{k},\delta)$. Then Proposition~\ref{prop T} provides a uniform bound on the first determinant in \eqref{3.1}, as well as on the second determinant when $z\in D\setminus \bigcup_{k=1}^N B(\mathcal{T}(x_{k}),\til\delta)$, where we pick $\til\delta=\til\delta(\Omega,\delta)\in(0,\frac 16)$ so that $\calT (\Omega\cap\bigcup_{k=1}^N B(x_{k},\delta))\supseteq D\cap \bigcup_{k=1}^N B(\mathcal{T}(x_{k}),3\til\delta)$ (the second bound will then depend on $\til\delta$). We therefore only need to estimate the integrals
 \[
 \int_{D} \frac{ (1- |z|) |\xi \cdot z^\perp | }{| \xi - z|^2 | |z|^2 \xi- z|^2 } \, dz \qquad\text{and}\qquad \int_{D\cap \bigcup_{k=1}^N B(\mathcal{T}(x_{k}),\til\delta)} \frac{ (1- |z|) |\xi \cdot z^\perp | }{| \xi - z|^2| |z|^2 \xi- z|^2 }\det D \mathcal{T}^{-1} (z) \, dz. 
 \]
If $z\in B(\mathcal{T}(x_{k}),\til\delta)$, then we have $| \xi - z|>2\til\delta$ because $\xi\in D\setminus B(\mathcal{T}(x_{k}),3\til\delta)$, and
\[
 | |z|^2 \xi- z|\geq |z|^2 \Big( |\xi -\mathcal{T}(x_{k}) | - | \mathcal{T}(x_{k}) - \tfrac{z}{|z|^2}|\Big)\geq (1-\tilde \delta)^2(3 \tilde \delta-\tfrac32 \tilde \delta) > \tilde \delta
\]
because $\tilde \delta<\frac 16$. This and the $D\mathcal{T}^{-1}$ bound from Proposition~\ref{prop T} estimate the second integral above by a constant depending only on $\Omega$ (through $\delta$ and $\til\delta$, which depend only on $\Omega$). 

If $\xi\in B(0, 1-\delta)$, then $| |z|^2 \xi- z|\ge (1-|z||\xi|)|z|\ge \delta|z|$. Since $|\xi \cdot z^\perp |\le |\xi-z| |z|$ and we only consider $\xi\in D\setminus B(0,\frac 12)$ in this lemma (so that $\max\{|\xi-z|,|z|\}\ge\frac 14$), it follows that the first integral above is also uniformly bounded in all $\xi\in B(0, 1-\delta)$. It therefore suffices to show that
\beq\lb{A.1}
 \int_{D} \frac{ (1- |z|) |\xi \cdot z^\perp | }{| \xi - z|^2 | |z|^2 \xi- z|^2 } \, dz \le C \left| \ln (1-| \xi|)\right|
\eeq
for $\xi\in D\setminus B(0, 1-\delta)$. From rotational symmetry of this integral in $\xi$ it follows that we only need to consider $\xi\in D\cap B(-e_1,\delta)$, with $e_1=(1,0)$, to finish the proof for $\calT^{-1}(\xi)\in \Omega\setminus \bigcup_{k=1}^N B(x_{k},\delta)$.

If instead $\calT^{-1}(\xi)\in B(x_{k},\delta)$ for some $k$, then our choice of $\delta$ shows that $\xi\in B(\calT(x_{k}),\delta)$, and so Proposition~\ref{prop T} and \eqref{est:DT1} show that we only need to prove
 \[
|\xi-\mathcal{T}(x_{k})|^{2\frac{\pi-\theta_{k}}{\pi}} \int_{D} \frac{ (1- |z|) |\xi \cdot z^\perp | }{| \xi - z|^2| |z|^2 \xi- z|^2 }|z-\mathcal{T}(x_{k})|^{2\frac{\theta_{k}-\pi}{\pi}} \, dz \le C \left| \ln (1-| \xi|)\right|. 
 \]
This is because the integral in \eqref{3.1} over $B(\calT(x_{j}),\delta)$ for each $j\neq k$ is obviously estimated above by a constant, using again $\min_{k} \theta_{k}>0$. We can now assume without loss that $\calT(x_k)=-e_1$. Thus we only need to show 
\beq\lb{A.2}
|\xi+e_1|^{2\frac{\pi-\theta}{\pi}} \int_{D} \frac{ (1- |z|) |\xi \cdot z^\perp | }{| \xi - z|^2| |z|^2 \xi- z|^2 }|z+e_1|^{2\frac{\theta-\pi}{\pi}} \, dz \le C \left| \ln (1-| \xi|)\right|
 \eeq
 for any $\xi\in D\cap B(-e_1,\delta)$ and any fixed $\theta\in(0,\pi]$ (so this includes also \eqref{A.1}), where $C$ may depend on $\theta$. 

Let $R_\xi:=\frac 14(1-|\xi|)$ and first consider the left-hand side of \eqref{A.2} with the integral only over $z\in B(\xi,R_\xi)$. The substitution $\eta:=z-\xi$ yields
\beq\lb{A.2a}
|\xi+e_1|^{2\frac{\pi-\theta}{\pi}} \int_{B(0,R_\xi)} \frac{ (1- |\eta+\xi|) |\xi \cdot \eta^\perp | }{| \eta|^2| (1-|\eta+\xi|^2) \xi+ \eta|^2 }|\eta+\xi+e_1|^{2\frac{\theta-\pi}{\pi}} \, d\eta.
\eeq
For $\eta \in B(0,R_{\xi})$ we have $|\eta|\leq \frac14(1-|\xi|)\leq \frac14|\xi+e_{1}|$, hence $|\xi+e_{1}|\leq |\eta+\xi+e_1|+ |\eta|$ yields $|\eta+\xi+e_1|\ge \frac 34|\xi+e_1|$. Since also $1- |\eta+\xi|\in(3R_\xi, 5R_\xi)$ and $|\xi|\ge 1-\delta\ge \frac 23$, \eqref{A.2a} is no more than 
\[
 \int_{B(0,R_\xi)} \frac{ 5R_\xi |\eta| }{| \eta|^2| 2R_\xi- |\eta||^2 } (4/3)^2 \, d\eta 
 \le 10 \int_{B(0,R_\xi)} \frac{ 1 }{ R_\xi| \eta| } \, d\eta =20\pi .
\]
Hence we only need to prove \eqref{A.2} with the integral over $z\in D\setminus B(\xi,R_\xi)$. 

For these $z$ we employ the estimate
\[
\left| |z|^2 \xi- z\right| = |z| \left| |z|\xi-\frac{z}{|z|}\right| \ge |z||\xi-z|,
\]
where in the inequality we used that the points $|z|\xi$ and $\frac{z}{|z|}$ lie on the same radii of $D$ as the points $\xi$ and $z$, respectively, but $||z|\xi|<|z|$ and $|\frac{z}{|z|}| > |\xi|$
(hence the distance of the former pair is larger). After also using $|\xi \cdot z^\perp |= |(\xi-z) \cdot z^\perp | \le|\xi-z||z|$, we are left with proving
\[
|\xi+e_1|^{2\frac{\pi-\theta}{\pi}} \int_{D\setminus B(\xi,R_\xi)} \frac{ 1- |z| }{| \xi - z|^3 |z|}|z+e_1|^{2\frac{\theta-\pi}{\pi}} \, dz \le C_\theta \left| \ln (1-| \xi|)\right|
 \]
 for all $\xi\in D\cap B(-e_1,\delta)$ and $\theta\in(0,\pi]$. This is obviously true if we restrict the integral to $z\in B(0,\frac 12)$, while on the rest of the domain the $|z|$ in the denominator can be neglected. So after we also shift both $\xi$ and $z$ to the right by 1 (and let $R_\xi':=\frac 14(1-|\xi-e_1|)$), it suffices to prove
\beq\lb{A.3}
|\xi|^{2\frac{\pi-\theta}{\pi}} \int_{B(e_1,1)\setminus B(\xi,R_\xi')} \frac{ 1- |z-e_1| }{| \xi - z|^3}|z|^{2\frac{\theta-\pi}{\pi}} \, dz \le C_\theta \left| \ln (1-| \xi-e_1|)\right|
 \eeq
for all $\xi\in B(e_1,1)\cap B(0,\delta)$ and $\theta\in(0,\pi]$. 
 
If we restrict the integral to $z\in B(0,\frac 12|\xi|)$ and use $| z| \ge 1-| z-e_1|$, the left-hand side will be no more than
\[
|\xi|^{2\frac{\pi-\theta}{\pi}} \int_{B(0,\frac 12|\xi|)} \frac{ |z| }{|\xi|^3/8}|z|^{2\frac{\theta-\pi}{\pi}} \, dz \le C. 
\]
If we restrict the integral in \eqref{A.3} to $z\notin B(\xi,\frac 12|\xi|)\cup B(0,\frac 12|\xi|)$ and use also $|\xi-z|\geq \frac13|z|$, the left-hand side will be no more than
\[
|\xi|^{2\frac{\pi-\theta}{\pi}} \int_{B(e_1,1)\setminus (B(\xi,\frac 12|\xi|)\cup B(0,\frac 12|\xi|))} \frac{ |z| }{|z|^3/27}|z|^{2\frac{\theta-\pi}{\pi}} \, dz \le \begin{cases} C(\pi-\theta)^{-1} & \theta\in(0,\pi), \\ C \left| \ln | \xi| \right| & \theta=\pi. \end{cases}
\]
If we restrict the integral in \eqref{A.3} to $z\in B(\xi,\frac 12|\xi|)$ and use that
\[
1- |z-e_1| \le |\xi-z| + (1- |\xi-e_1|) = |\xi-z|+4R_\xi' \le 5|\xi-z|
\]
for $z\in B(e_1,1)\setminus B(\xi,R_\xi')$, the left-hand side will be no more than
\[
\int_{B(\xi, \frac 12|\xi|)\setminus B(\xi,R_\xi')} \frac{ 5 }{| \xi - z|^2} 4 \, dz \le C \left| \ln R_\xi' \right|.
\]
Since $4R_\xi'=1-| \xi-e_1|\le |\xi|\le\delta\le\frac 13$, we obtain \eqref{A.3} and the proof is finished.

\def\cprime{$'$}

\adrese


\begin{thebibliography}{10}

\bibitem{Adams}
R.~A. Adams and J.~J.~F. Fournier.
\newblock {\em Sobolev spaces}, volume 140 of {\em Pure and Applied Mathematics
  (Amsterdam)}.
\newblock Elsevier/Academic Press, Amsterdam, second edition, 2003.

\bibitem{Ahlfors}
L.~V. Ahlfors.
\newblock {\em Lectures on quasiconformal mappings}.
\newblock Manuscript prepared with the assistance of Clifford J. Earle, Jr. Van
  Nostrand Mathematical Studies, No. 10. D. Van Nostrand Co., Inc., Toronto,
  Ont.-New York-London, 1966.

\bibitem{Ambrosio}
L.~Ambrosio.
\newblock Transport equation and {C}auchy problem for non-smooth vector fields.
\newblock In {\em Calculus of variations and nonlinear partial differential
  equations}, volume 1927 of {\em Lecture Notes in Math.}, pages 1--41.
  Springer, Berlin, 2008.

\bibitem{Bardos}
C.~Bardos.
\newblock Existence et unicit\'e de la solution de l'\'equation d'{E}uler en
  dimension deux.
\newblock {\em J. Math. Anal. Appl.}, 40:769--790, 1972.

\bibitem{BDT}
C.~Bardos, F.~Di~Plinio, and R.~Temam.
\newblock The {E}uler equations in planar nonsmooth convex domains.
\newblock {\em J. Math. Anal. Appl.}, 407(1):69--89, 2013.

\bibitem{Kondra}
M.~Borsuk and V.~Kondratiev.
\newblock {\em Elliptic boundary value problems of second order in piecewise
  smooth domains}, volume~69 of {\em North-Holland Mathematical Library}.
\newblock Elsevier Science B.V., Amsterdam, 2006.

\bibitem{CDGG}
J.-Y. Chemin, B.~Desjardins, I.~Gallagher, and E.~Grenier.
\newblock {\em Mathematical geophysics}, volume~32 of {\em Oxford Lecture
  Series in Mathematics and its Applications}.
\newblock The Clarendon Press, Oxford University Press, Oxford, 2006.
\newblock An introduction to rotating fluids and the Navier-Stokes equations.

\bibitem{Delort}
J.-M. Delort.
\newblock Existence de nappes de tourbillon en dimension deux.
\newblock {\em J. Amer. Math. Soc.}, 4(3):553--586, 1991.

\bibitem{DT}
F.~Di~Plinio and R.~Temam.
\newblock Grisvard's {S}hift {T}heorem {N}ear {$L^\infty$} and {Y}udovich
  {T}heory on {P}olygonal {D}omains.
\newblock {\em SIAM J. Math. Anal.}, 47(1):159--178, 2015.

\bibitem{DiPernaLions}
R.~J. DiPerna and P.-L. Lions.
\newblock Ordinary differential equations, transport theory and {S}obolev
  spaces.
\newblock {\em Invent. Math.}, 98(3):511--547, 1989.

\bibitem{DiPernaMajda}
R.~J. DiPerna and A.~J. Majda.
\newblock Concentrations in regularizations for {$2$}-{D} incompressible flow.
\newblock {\em Comm. Pure Appl. Math.}, 40(3):301--345, 1987.

\bibitem{Galdi}
G.~P. Galdi.
\newblock {\em An introduction to the mathematical theory of the
  {N}avier-{S}tokes equations}.
\newblock Springer Monographs in Mathematics. Springer, New York, second
  edition, 2011.
\newblock Steady-state problems.

\bibitem{GV-Lac}
D.~G{\'e}rard-Varet and C.~Lacave.
\newblock The {T}wo-{D}imensional {E}uler {E}quations on {S}ingular {D}omains.
\newblock {\em Arch. Ration. Mech. Anal.}, 209(1):131--170, 2013.

\bibitem{GV-Lac2}
D.~G{\'e}rard-Varet and C.~Lacave.
\newblock The {T}wo {D}imensional {E}uler {E}quations on {S}ingular {E}xterior
  {D}omains.
\newblock {\em Arch. Ration. Mech. Anal.}, 218(3):1609--1631, 2015.

\bibitem{Grisvard}
P.~Grisvard.
\newblock {\em Elliptic problems in nonsmooth domains}, volume~24 of {\em
  Monographs and Studies in Mathematics}.
\newblock Pitman (Advanced Publishing Program), Boston, MA, 1985.

\bibitem{Gustafsson}
B.~Gustafsson and A.~Vasil{\cprime}ev.
\newblock {\em Conformal and potential analysis in {H}ele-{S}haw cells}.
\newblock Advances in Mathematical Fluid Mechanics. Birkh\"auser Verlag, Basel,
  2006.

\bibitem{Kenig}
D.~Jerison and C.~E. Kenig.
\newblock The inhomogeneous {D}irichlet problem in {L}ipschitz domains.
\newblock {\em J. Funct. Anal.}, 130(1):161--219, 1995.

\bibitem{Kikuchi}
K.~Kikuchi.
\newblock Exterior problem for the two-dimensional {E}uler equation.
\newblock {\em J. Fac. Sci. Univ. Tokyo Sect. IA Math.}, 30(1):63--92, 1983.

\bibitem{KisZla}
A.~Kiselev and A.~Zlato\v{s}.
\newblock Blow up for the 2{D} {E}uler equation on some bounded domains.
\newblock {\em J. Differential Equations}, 259(7):3490--3494, 2015.

\bibitem{Mazya}
V.~A. Kozlov, V.~G. Maz{\cprime}ya, and J.~Rossmann.
\newblock {\em Spectral problems associated with corner singularities of
  solutions to elliptic equations}, volume~85 of {\em Mathematical Surveys and
  Monographs}.
\newblock American Mathematical Society, Providence, RI, 2001.

\bibitem{Lacave-SIAM}
C.~Lacave.
\newblock Uniqueness for two-dimensional incompressible ideal flow on singular
  domains.
\newblock {\em SIAM J. Math. Anal.}, 47(2):1615--1664, 2015.

\bibitem{lacave-miot}
C.~Lacave and E.~Miot.
\newblock Uniqueness for the vortex-wave system when the vorticity is constant
  near the point vortex.
\newblock {\em SIAM J. Math. Anal.}, 41(3):1138--1163, 2009.

\bibitem{LMW}
C.~Lacave, E.~Miot, and C.~Wang.
\newblock Uniqueness for the two-dimensional {E}uler equations on domains with
  corners.
\newblock {\em Indiana Univ. Math. J.}, 63(6):1725--1756, 2014.

\bibitem{marchioro}
C.~Marchioro.
\newblock On the {E}uler equations with a singular external velocity field.
\newblock {\em Rend. Sem. Mat. Univ. Padova}, 84:61--69 (1991), 1990.

\bibitem{MarPul-VWS}
C.~Marchioro and M.~Pulvirenti.
\newblock On the vortex–wave system.
\newblock In M.~Francaviglia, editor, {\em Mechanics, Analysis and Geometry:
  200 Years After Lagrange}, North-Holland Delta Series, pages 79 -- 95.
  Elsevier, Amsterdam, 1991.

\bibitem{MarPul}
C.~Marchioro and M.~Pulvirenti.
\newblock {\em Mathematical theory of incompressible nonviscous fluids},
  volume~96 of {\em Applied Mathematical Sciences}.
\newblock Springer-Verlag, New York, 1994.

\bibitem{MazyaSolovev}
V.~Maz’ya and A.~Solov’ev.
\newblock On the boundary integral equation of the neumann problem in a domain
  with a peak.
\newblock {\em Amer. Math. Soc. Transl}, 155:101--127, 1993.

\bibitem{McGrath}
F.~J. McGrath.
\newblock Nonstationary plane flow of viscous and ideal fluids.
\newblock {\em Arch. Rational Mech. Anal.}, 27:329--348, 1967.

\bibitem{Pomm2}
C.~Pommerenke.
\newblock {\em Boundary behaviour of conformal maps}, volume 299 of {\em
  Grundlehren der Mathematischen Wissenschaften [Fundamental Principles of
  Mathematical Sciences]}.
\newblock Springer-Verlag, Berlin, 1992.

\bibitem{Taylor}
M.~E. Taylor.
\newblock Incompressible fluid flows on rough domains.
\newblock In {\em Semigroups of operators: theory and applications ({N}ewport
  {B}each, {CA}, 1998)}, volume~42 of {\em Progr. Nonlinear Differential
  Equations Appl.}, pages 320--334. Birkh\"auser, Basel, 2000.

\bibitem{Temam}
R.~Temam.
\newblock On the {E}uler equations of incompressible perfect fluids.
\newblock {\em J. Functional Analysis}, 20(1):32--43, 1975.

\bibitem{Vis1}
M.~Vishik.
\newblock Instability and non-uniqueness in the {C}auchy problem for the {E}uler equations of an ideal incompressible fluid. Part I.
\newblock{preprint.}

\bibitem{Vis2}
M.~Vishik.
\newblock Instability and non-uniqueness in the {C}auchy problem for the {E}uler equations of an ideal incompressible fluid. Part II.
\newblock{preprint.}

\bibitem{Wolibner}
W.~Wolibner.
\newblock Un theor\`eme sur l'existence du mouvement plan d'un fluide parfait,
  homog\`ene, incompressible, pendant un temps infiniment long.
\newblock {\em Math. Z.}, 37(1):698--726, 1933.

\bibitem{yudo}
V.~I. Yudovi{\v{c}}.
\newblock Non-stationary flows of an ideal incompressible fluid.
\newblock {\em \u Z. Vy\v cisl. Mat. i Mat. Fiz.}, 3:1032--1066, 1963.

\end{thebibliography}
\end{document}